\documentclass[12pt]{amsart}
\usepackage{amsmath, amsfonts, amssymb, amsthm,hyperref,mathtools,array}
\hypersetup{hypertex=true,
	colorlinks=true,
	linkcolor=blue,
	anchorcolor=blue,
	citecolor=blue}
\usepackage{bm}
\allowdisplaybreaks[4]
\def\({\bg(}
\def\){\bg)}

\def\rank #1{{\rm rank}\ #1}

\def\Tr{{\rm Tr}}

\def\u{{\bm u}}

\def\0{{\bm 0}}
\def\1{{\bm 1}}
\def\i{{\bm i}}

\def\diag{{\rm diag}}

\def\pmod #1{\ ({\rm{mod}}\ #1)}

\theoremstyle{plain}
\newtheorem{theorem}{Theorem}[section]
\newtheorem{lemma}{Lemma}

\theoremstyle{definition}

\theoremstyle{remark}
\newtheorem{remark}{Remark}

\newcommand{\sign}[1]{\mathrm{sign}(#1)}
\makeatletter
\@namedef{subjclassname@2020}{%
	\textup{2020} Mathematics Subject Classification}
\makeatother
\vspace{4mm}

\begin{document}
	
	\title[Cyclotomic matrices related to Kloosterman sums over finite fields]
	{Cyclotomic matrices related to Kloosterman sums over finite fields}
	\author[H.-L. Wu]{Hai-Liang Wu }
	
	\address {(Hai-Liang Wu) School of Science, Nanjing University of Posts and Telecommunications, Nanjing 210023, People's Republic of China}
	\email{\tt whl.math@smail.nju.edu.cn}
	
	\keywords{Kloosterman sums, cyclotomic matrices, finite fields.
		\newline \indent This research was supported by the Natural Science Foundation of China (Grant No. 12101321) and the Natural Science Foundation of the Higher Education Institutions of Jiangsu Province (Grant No. 25KJB110010).
	}
	
	\begin{abstract}
		In this paper, by using the arithmetic properties of character sums over finite fields, we investigate some cyclotomic matrices involving Kloosterman sums over finite fields. For example, let 
		$$K_q(u)=\sum_{x\in\mathbb{F}_q\setminus\{0\}}e^{\frac{2\pi i}{p}{\rm Tr}_{\mathbb{F}_q/\mathbb{F}_p}\left(x+\frac{u}{x}\right)}$$
		be the Kloosterman sum over $\mathbb{F}_q$, where $q=p^f$ is an odd prime power. We prove that matrix $[K_q(s_i+s_j)]_{1\le i,j\le (q-1)/2}$ is singular whenever $q\ge 11$, where $s_1,s_2,\cdots,s_{(q-1)/2}$ are exactly all non-zero squares over $\mathbb{F}_q$. 
	\end{abstract}
	\maketitle
	
	\section{Introduction}
	\setcounter{lemma}{0}
	\setcounter{theorem}{0}
	\setcounter{equation}{0}
	\setcounter{conjecture}{0}
	\setcounter{remark}{0}
	\setcounter{corollary}{0}
	
	\subsection{Notation}
	
	Throughout this paper, let $q=p^f$ be an odd prime power, where $p$ is an odd prime and $f\in\mathbb{Z}^+$. Let 
	$$\mathbb{F}_q=\left\{0=x_0,x_1,\cdots,x_{q-1}\right\}$$ 
	indicate the finite field with $q$ elements. Let $\mathbb{F}_q^{\times}=\mathbb{F}_q\setminus\{0\}$ be the multiplicative group of all non-zero elements over $\mathbb{F}_q$. We use $\widehat{\mathbb{F}_q^{\times}}$ to denote the cyclic group of all multiplicative characters of $\mathbb{F}_q$, and let $\chi_q$ be a generator of $\widehat{\mathbb{F}_q^{\times}}$. For any multiplicative character $\chi_q^r: \mathbb{F}_q^{\times}\rightarrow\mathbb{C}$, we additionally define $\chi_q^r(0)=0$. Let 
	$$\mathcal{S}_q=\left\{x^2: x\in\mathbb{F}_q^{\times}\right\}=\left\{s_1,s_2,\cdots,s_{(q-1)/2}\right\}$$
	and $\mathcal{N}_q=\mathbb{F}_q^{\times}\setminus\mathcal{S}_q$. Then the unique quadratic multiplicative character $\chi_q^{(q-1)/2}$ of $\mathbb{F}_q$ is denoted by $\phi_q$, that is, 
	$$\phi_q(x)=\begin{cases}
		1   & \mbox{if}\ x\in\mathcal{S}_q,\\
		0  & \mbox{if}\ x=0,\\
		-1 & \mbox{if}\ x\in\mathcal{N}_q.
	\end{cases}$$
	
	Let $\zeta_p=e^{2\pi \i/p}$ be a primitive $p$-th root of unity, where $\i=\sqrt{-1}$ with argument $\pi/2$. Let $\Tr: \mathbb{F}_q\rightarrow \mathbb{F}_p$ be the trace map, i.e., 
	$$\Tr(x)=x+x^{p}+\cdots+x^{p^{f-1}}$$
	for any $x\in\mathbb{F}_q$. Given any $a,b\in\mathbb{F}_q$, the Kloosterman sum is defined by 
	\begin{equation}\label{Eq. definition of Kloosterman sums}
		K_q(a,b)=\sum_{x\in\mathbb{F}_q^{\times}}\zeta_p^{\Tr(ax+\frac{b}{x})}.
	\end{equation}
	When $a\in\mathbb{F}_q^{\times}$, it is easy to see that $K_q(a,b)=K_q(1,ab)$. From now on, we abbreviate $K_q(1,b)$ as $K_q(b)$ for any $b\in\mathbb{F}_q$. 
	
	On the other hand, for any square matrix $M$, let $M(i,j)$ be the $(i,j)$-entry of $M$. Also, we use the symbols $\det M$ and $\rank M$ to denote the determinant and rank of $M$ respectively. 
	
	\subsection{Background and motivation} Kloosterman sums have extensive applications in number theory, modular forms and algebraic geometry. For example, in 1948, Weil \cite{Weil} used the theory of algebraic curves over finite fields to prove that 
	$$\left|K_q(a,b)\right|\le 2\sqrt{q}$$
	whenever $ab\in\mathbb{F}_q^{\times}$. In addition, when $q=p$ is an odd prime, let 
	$$h(z)=\eta(z)^2\eta(2z)^2\eta(3z)^2\eta(6z)^2$$ 
	be a cusp form of weight $4$, where $\eta(z)$ is the Dedekind eta function. Hulek, Spandaw, van Geemen and van Straten \cite{H} showed that 
	$$\sum_{b=0}^{p-1}K_p(b)^6=5p^4-10p^3-9p^2-c_pp^2-5p,$$
	where $c_p$ is the $p$-th fourier coefficient of the cusp form $h(z)$. For more results on $k$-th power moments of Kloosterman sums, readers may refer to \cite{Evans}. 
	
	Next we state our research motivation.  Carlitz \cite{Carlitz} initiated the investigation of cyclotomic matrices. For example, Carlitz \cite[Theorem 5]{Carlitz} proved that 
	$$\det \left[\phi_p(i+j)\right]_{1\le i,j\le p-1}=\frac{1}{p}G_p^{p-1}=(-1)^{(p-1)/2}p^{(p-3)/2},$$
	where 
	$$G_p=\sum_{x\in\mathbb{F}_p}\phi_p(x)\zeta_p^x$$
	is the quadratic Gauss sum over $\mathbb{F}_p$. Along this line, Z.-W. Sun \cite[Theorem 1.2]{Sun} showed that  
	 $$\det \left[\phi_p(i^2+j^2)\right]_{1\le i,j\le (p-1)/2},$$
	 is a quadratic residue modulo $p$. Inspired by the above results due to Carlitz and Sun, recently, the author, Li, Wang and Yip \cite{WLWY} investigated some determinants involving Gauss sums, and proved that 
	 $$\det\left[G_q(\chi_q^{i+j})\right]_{0\le i,j\le q-2}=(-1)^{(q-3)/2}(q-1)^{q-1}$$
	 and 
	 $$\det\left[G_q(\chi_q^{2i+2j})\right]_{0\le i,j\le (q-3)/2}=(-1)^{\alpha_f}\cdot \left(\frac{q-1}{2}\right)^{(q-1)/2}\cdot  2^{(p^{f-1}-1)/2},$$
	 where $f=[\mathbb{F}_q:\mathbb{F}_p]$ and 
	 $$\alpha_f=\begin{cases}
	 	1                    &  \mbox{if}\ f\equiv 1\pmod 2,\\
	 	(p^2+7)/8  &  \mbox{if}\ f\equiv 0\pmod 2.
	 \end{cases}$$
	The above two determinants can be viewed as the finite field analogues of the following determinant \cite[(4.5)]{Normand}:
	$$\det \left[\Gamma(i+j)\right]_{1\le i,j\le n}=\prod_{r=0}^{n-1}r!(r+1)!,$$
	where $\Gamma(\cdot)$ is the Gamma function. Now we return to the Kloosterman sums. The Kloosterman sum $K_p(a,b)$ is a finite field analogue of certain Bessel function. In fact, the modified Bessel function of the second kind is defined by 
	$$K_0(x)=\frac{1}{2}\int_{0}^{\infty}\exp\left(-\frac{x}{2}\left(t+\frac{1}{t}\right)\right)\frac{dt}{t},$$
	where $x>0$. For any positive real numbers $\alpha$ and $\beta$, using integral transforms, we can obtain 
	\begin{equation}\label{Eq. Bessel function}
		K_0(2\sqrt{\alpha\beta})=\frac{1}{2}\int_{0}^{\infty}\exp\left(-\alpha t-\frac{\beta}{t}\right)\frac{dt}{t}.
	\end{equation}
	Thus, (\ref{Eq. definition of Kloosterman sums}) can be viewed as a finite field analogue of (\ref{Eq. Bessel function}). For a concrete introduction on finite field analogues, readers may refer to \cite{F,G}. 
	
	Motivated by the above works, in this paper, we focus on the cyclotomic matrices related to Kloosterman sums over $\mathbb{F}_q$. Specifically, we concentrate on the matrices 
	$$A_q(1)=\left[K_q(x_i+x_j)\right]_{1\le i,j\le q-1},$$
	and
	$$A_q(2)=\left[K_q(s_i+s_j)\right]_{1\le i,j\le (q-1)/2}.$$

	\subsection{Main theorems} Now we state our first theorem.
	
	\begin{theorem}\label{Thm. A}
		Let $q$ be an odd prime power. Then $\det A_q(1)=(-1)^{(q-1)/2}q^{q-2}$.
	\end{theorem}
	
	\begin{remark}\label{Remark of Thm. A}
		The proof of Theorem \ref{Thm. A} is based on standard matrix decomposition of $A_q(1)$.
	\end{remark}
	
	Next we state our second theorem.
	
	\begin{theorem}\label{Thm. B}
		Let $q$ be an odd prime power. Then the matrix $A_q(2)$ is singular whenever $q\ge 11$. In addition, $\det A_3(2)=2, \det A_5(2)=-5$, and $\det A_7(2)=49$. 
	\end{theorem}
	
	\begin{remark}\label{Remark of Thm. B}
		The proof of Theorem \ref{Thm. B} is much more difficult than that of Theorem \ref{Thm. A}. We will use the quadratic Gauss sum over $\mathbb{F}_q$ to construct a  complicated additive decomposition of $A_q(2)$.
	\end{remark}
	
    \subsection{Outline of this paper} We will prove our theorems in Sections 2--3 respectively.

	\section{Proof of Theorem \ref{Thm. A}}
	\setcounter{lemma}{0}
	\setcounter{theorem}{0}
	\setcounter{equation}{0}
	\setcounter{conjecture}{0}
	\setcounter{remark}{0}
	\setcounter{corollary}{0}
	
	Recall that $q$ is an odd prime power and that  $\mathbb{F}_q^{\times}=\{x_1,x_2,\cdots,x_{q-1}\}$. We begin with the following lemma.
	
	\begin{lemma}\label{Lem. permutation}
		Let $q$ be an odd prime and $\tau_q$ be a permutation of $\mathbb{F}_q^{\times}$ defined by $x\mapsto -x$ for any $x\in\mathbb{F}_q^{\times}$. Then 
		$$\sign{\tau_q}=(-1)^{\frac{q-1}{2}},$$
		where $\sign{\tau_q}$ denotes the sign of $\tau_q$. 
	\end{lemma}
	
	\begin{proof}
		As $q$ is an odd prime power, the integers $\pm 1$ can be viewed as two distinct elements of $\mathbb{F}_q$. Thus, 
		$$\sign{\tau_q}=\prod_{1\le i<j\le q-1}-\frac{x_j-x_i}{x_j-x_i}=(-1)^{\frac{(q-1)(q-2)}{2}}=(-1)^{\frac{q-1}{2}}.$$
		This completes the proof. 
	\end{proof}
	
	Now we are in a position to prove our first theorem.
	
	{\noindent\bfseries Proof of Theorem \ref{Thm. A}}. Define the $(q-1)\times (q-1)$ matrix $M_q$ by 
	$$M_q(i,j)=\zeta_p^{\Tr(x_ix_j)}$$
	for any integers $i,j\in[1,q-1]$. Let the $(q-1)\times (q-1)$ diagonal matrix 
	$$D_q=\diag\left(\zeta_p^{\Tr(x_1^{-1})}, \zeta_p^{\Tr(x_2^{-1})}, \cdots, \zeta_p^{\Tr(x_{q-1}^{-1})}\right).$$
	
	We first evaluate $\det M_q^{H} M_q$, where $M_q^H$ denotes the conjugate transpose (or Hermitian transpose) of $M_q$, i.e., $M_q^H=[\zeta_p^{-\Tr(x_ix_j)}]_{1\le i,j\le q-1}$. Given any integers $i,j\in[1,q-1]$, one can verify that 
	\begin{align}\label{Eq. a in the proof of Thm. A}
		M_q^H M_q(i,j)
		&=\sum_{k=1}^{q-1}M_q^H(i,k)M_q(k,j)\notag\\
		&=\sum_{k=1}^{q-1}\overline{M_q(k,i)}M_q(k,j)\notag\\
		&=\sum_{k=1}^{q-1}\zeta_p^{-\Tr(x_kx_i)}\zeta_p^{\Tr(x_kx_j)}\notag\\
		&=\sum_{x\in\mathbb{F}_q^{\times}}\zeta_p^{\Tr\left((x_j-x_i)x_k\right)}.
	\end{align}
	Since $\Tr: \mathbb{F}_q\rightarrow\mathbb{F}_p$ is a surjective $\mathbb{F}_p$-homomorphism with $\ker(\Tr)=\left\{x^p-x:\ x\in\mathbb{F}_q\right\}$, one can verify that 
	\begin{equation}\label{Eq. b in the proof of Thm. A}
		\sum_{x\in\mathbb{F}_q^{\times}}\zeta_p^{\Tr(x)}=-1+\sum_{x\in\mathbb{F}_q}\zeta_p^{\Tr(x)}=-1+\#\ker(\Tr)\sum_{y\in\mathbb{F}_p}\zeta_p^y=-1+\#\ker(\Tr)\sum_{y=0}^{p-1}\zeta_p^y=-1,
	\end{equation}
	where $\# S$ denotes the cardinality of a finite set $S$. By (\ref{Eq. a in the proof of Thm. A}) and (\ref{Eq. b in the proof of Thm. A}) we obtain 
	$$M_q^H M_q(i,j)=
	\begin{cases}
		q-1  &  \mbox{if}\ 1\le i=j\le q-1,\\
		-1    &   \mbox{if}\ 1\le i\neq j\le q-1.
	\end{cases}$$
	Thus, 
	$$M_q^{H}M_q=qI_{q-1}-J_{q-1},$$
	where $I_{q-1}$ is the $(q-1)\times (q-1)$ identity matrix and $J_{q-1}$ is a $(q-1)\times (q-1)$ matrix with all entries $1$. Since the numbers $q-1, 0, 0, \cdots, 0$ are exactly all the eigenvalues of $J_{q-1}$, we see that $1, q, \cdots, q$ are precisely all the eigenvalues of $M_q^{H}M_q$. Therefore, 
	\begin{equation}\label{Eq. c in the proof of Thm. A}
		\det M_q^{H}M_q=q^{q-2}
	\end{equation}
	
	Next we compute $\det D_q$. Noting that
	\begin{equation}\label{Eq. sum of trace is 0}
		\sum_{x\in\mathbb{F}_q^{\times}}\Tr(x)=\Tr\left(\sum_{x\in\mathbb{F}_q^{\times}}x\right)=\Tr(0)=0,
	\end{equation}
	we obtain 
	\begin{equation}\label{Eq. d in the proof of Thm. A}
		\det D_q=\prod_{k=1}^{q-1}\zeta_p^{\Tr(x_k^{-1})}=\prod_{x\in\mathbb{F}_q^{\times}}\zeta_p^{\Tr(x)}=1.
	\end{equation}
	
	Now we consider $\det [K_q(x_i-x_j)]_{1\le i,j\le q-1}$. As 
	$$K_q(x_i-x_j)=\sum_{k=1}^{q-1}\zeta_p^{\Tr\left(x_k+\frac{x_i-x_j}{x_k}\right)}
	=\sum_{k=1}^{q-1}\zeta_p^{\Tr\left((x_i-x_j)x_k+\frac{1}{x_k}\right)}$$
	for any integers $i,j\in[1,q-1]$. Thus, we obtain the decomposition
	$$\left[K_q(x_i-x_j)\right]_{1\le i,j\le q-1}=M_qD_qM_q^{H}.$$
	Assembling this and (\ref{Eq. c in the proof of Thm. A})--(\ref{Eq. d in the proof of Thm. A}), we have 
	\begin{equation}\label{Eq. e in the proof of Thm. A}
		\det \left[K_q(x_i-x_j)\right]_{1\le i,j\le q-1}=\det M_qD_qM_q^{H}=\det M_q^{H}M_q\cdot \det D_q=q^{q-2}.
	\end{equation}
	
	Finally, we determine $\det A_q(1)$. By Lemma \ref{Lem. permutation}, we see that 
    $$\det A_q(1)=\sign{\tau_q}\cdot \det \left[K_q(x_i-x_j)\right]_{1\le i,j\le q-1}=(-1)^{(q-1)/2}q^{q-2}.$$
	
	In view of the above, we have completed the proof. \qed 

    \section{Proof of Theorem \ref{Thm. B}}
    \setcounter{lemma}{0}
    \setcounter{theorem}{0}
    \setcounter{equation}{0}
    \setcounter{conjecture}{0}
    \setcounter{remark}{0}
    \setcounter{corollary}{0}

   Recall that $q=2n+1$ is an odd prime power,  $\mathbb{F}_q^{\times}=\left\{x_1,x_2,\cdots,x_{q-1}\right\}$, 
   $$\mathcal{S}_q=\left\{x^2:\ x\in\mathbb{F}_q^{\times}\right\}=\left\{s_1,s_2,\cdots,s_n\right\},$$
   and $\mathcal{N}_q=\mathbb{F}_q^{\times}\setminus\mathcal{S}_q$. 
   Let 
   \begin{equation}\label{Eq. definition of Hq}
   	 \mathcal{H}_q=\left\{y_1,y_2,\cdots,y_n\right\}\subseteq\mathbb{F}_q^{\times}
   \end{equation}
   such that $y_k^2=s_k$ for any integer $k\in[1,n]$. By the definition of $\mathcal{H}_q$, we clearly have 
   $$\left\{y:\ y\in\mathcal{H}_q\right\}\cup\left\{-y:\ y\in\mathcal{H}_q\right\}=\mathbb{F}_q^{\times}.$$

   We begin with the following lemma.
   
   \begin{lemma}\label{Eq. number of solutions of an equation over finite field}
   	Let $q=2n+1$ be an odd prime power. Then 
   	$$\#\left\{(x,y)\in\mathcal{H}_q\times\mathcal{H}_q:\ x^2+y^2=1\right\}=\frac{1}{4}\left(q-4+(-1)^{n+1}\right).$$
   	
   \end{lemma}

   \begin{proof}
   	By symmetry one can verify that 
   	\begin{align*}
   	&\#\left\{(x,y)\in\mathcal{H}_q\times\mathcal{H}_q:\ x^2+y^2=1\right\}\\
  =&\frac{1}{4}\#\left\{(x,y)\in\mathbb{F}_q^{\times}\times\mathbb{F}_q^{\times}:\ x^2+y^2=1\right\}\\
  =&\frac{1}{4}\#\left\{(x,y)\in\mathbb{F}_q\times\mathbb{F}_q:\ x^2+y^2=1\right\}-1\\
  =&\frac{1}{4}\sum_{x\in\mathbb{F}_q}\left(1+\phi_q(1-x^2)\right)-1\\
  =&\frac{q-4}{4}+\frac{\phi_q(-1)}{4}\sum_{x\in\mathbb{F}_q}\phi_q(x^2-1)\\
  =&\frac{1}{4}\left(q-4+(-1)^{n+1}\right),
   	\end{align*}
   	where the last equality follows from the following known result (see \cite[Theorem 5.48]{Lidl})
   	$$\sum_{x\in\mathbb{F}_q}\phi_q(x^2-1)=-1.$$
   	
   	In view of the above, we have completed the proof. 
   \end{proof}

   Now we define an $n\times n$ symmetric matrix $N_q$ by 
   \begin{equation}\label{Eq. definition of Nq}
   	N_q(i,j)=\zeta_p^{\Tr(2y_iy_j)}+\zeta_p^{\Tr(-2y_iy_j)}
   \end{equation}
   for any integers $i,j\in [1,n]$. We need the following result.
   
   \begin{lemma}\label{Lem. Nq is an invertible matrix}
   	 Let $q=2n+1$ be an odd prime power. Then $N_q$ is an invertible matrix. 
   \end{lemma}

   \begin{proof}
   	Let $N_q^T$ be the transpose of $N_q$. For any integers $i,j\in[1,n]$, recalling that 
   	$$\left\{y:\ y\in\mathcal{H}_q\right\}\cup\left\{-y:\ y\in\mathcal{H}_q\right\}=\mathbb{F}_q^{\times},$$
   	one can verify that 
   	\begin{align}\label{Eq. a in the proof of Lem. Nq is an invertible matrix}
   		N_qN_q^T(i,j)
 &=\sum_{k=1}^n\left(\zeta_p^{\Tr(2y_iy_k)}+\zeta_p^{-\Tr(2y_iy_k)}\right)\left(\zeta_p^{\Tr(2y_jy_k)}+\zeta_p^{-\Tr(2y_jy_k)}\right)\notag\\
 &=\sum_{k=1}^n\left(\zeta_p^{\Tr(2(y_i+y_j)y_k)}+\zeta_p^{\Tr(-2(y_i+y_j)y_k)}\right)+\sum_{k=1}^n\left(\zeta_p^{\Tr(2(y_i-y_j)y_k)}+\zeta_p^{\Tr(-2(y_i-y_j)y_k)}\right)\notag\\
 &=\sum_{x\in\mathbb{F}_q^{\times}}\zeta_p^{\Tr(2(y_i+y_j)x)}+\sum_{x\in\mathbb{F}_q^{\times}}\zeta_p^{\Tr(2(y_i-y_j)x)}\notag\\
 &=\sum_{x\in\mathbb{F}_q^{\times}}\zeta_p^{\Tr((y_i+y_j)x)}+\sum_{x\in\mathbb{F}_q^{\times}}\zeta_p^{\Tr((y_i-y_j)x)}.
   	\end{align}
   	Observe that $y+y'\neq 0\in\mathbb{F}_q$ for any $y,y'\in\mathcal{H}_q$. Hence, for any integers $i,j\in[1,n]$, by (\ref{Eq. b in the proof of Thm. A}) we obtain 
   	\begin{equation}\label{Eq. b in the proof of Lem. Nq is an invertible matrix}
   		\sum_{x\in\mathbb{F}_q^{\times}}\zeta_p^{\Tr((y_i+y_j)x)}=\sum_{x\in\mathbb{F}_q^{\times}}\zeta_p^{\Tr(x)}=-1,
   	\end{equation}
   	and 
   	\begin{equation}\label{Eq. c in the proof of Lem. Nq is an invertible matrix}
   		\sum_{x\in\mathbb{F}_q^{\times}}\zeta_p^{\Tr((y_i-y_j)x)}=
   		\begin{cases}
   			q-1  &  \mbox{if}\ i=j,\\
   			-1    &  \mbox{if}\ i\neq j.
   		\end{cases}
   	\end{equation}
   	Assembling (\ref{Eq. a in the proof of Lem. Nq is an invertible matrix})--(\ref{Eq. c in the proof of Lem. Nq is an invertible matrix}) gives 
   	$$N_qN_q^T=qI_n-2J_n,$$
   	where $I_n$ is the identity matrix of order $n$, and $J_n$ is an $n\times n$ matrix of all entries $1$. Noting that $n, 0, \cdots, 0$ are exactly all the eigenvalues of $J_n$, we see that $1, q,\cdots,q$ are precisely all the eigenvalues of $N_qN_q^T$. Thus, 
   	$$(\det N_q)^2=\det N_qN_q^T=q^{n-1}\neq 0.$$
   	This shows that $N_q$ is invertible.
   	
   	In view of the above, we have completed the proof. 
   \end{proof}

   Next we turn to the quadratic Gauss sum 
   \begin{equation}\label{Eq. definition of quadratic gauss sums}
   G_q=\sum_{x\in\mathbb{F}_q}\phi_q(x)\zeta_p^{\Tr(x)}.
\end{equation}
  By (\ref{Eq. b in the proof of Thm. A}) we see that
  $$\sum_{x\in\mathbb{F}_q}\zeta_p^{\Tr(x)}=0.$$
  Thus, we have 
  \begin{equation}\label{Eq. quadratic gauss sums with Sq}
  	G_q=\sum_{x\in\mathbb{F}_q}\left(1+\phi_q(x)\right)\zeta_p^{\Tr(x)}=\sum_{x\in\mathbb{F}_q}\zeta_p^{\Tr(x^2)}=1+2\sum_{x\in\mathcal{S}_q}\zeta_p^{\Tr(x)},
  \end{equation}
  and 
  \begin{equation}\label{Eq. quadratic gauss sums with Nq}
  	G_q=\sum_{x\in\mathbb{F}_q}\left(-1+\phi_q(x)\right)\zeta_p^{\Tr(x)}=-1-2\sum_{x\in\mathcal{N}_q}\zeta_p^{\Tr(x)},
  \end{equation}
  where $\mathcal{S}_q=\left\{x^2:\ x\in\mathbb{F}_q^{\times}\right\}$ and $\mathcal{N}_q=\mathbb{F}_q^{\times}\setminus\mathcal{S}_q$. In 1811, Gauss  first obtained the explicit value of quadratic Gauss sum over $\mathbb{F}_p$, that is, 
   $$\sum_{x\in\mathbb{F}_p}\phi_p(x)\zeta_p^x=\sqrt{(-1)^{(p-1)/2}p}.$$
   Using the Hasse-Davenport lifting formula (cf. \cite[Theorem 3.7.4]{Cohen}), one can generalize the above result to $\mathbb{F}_q$, where $q=p^f$. More specifically, we have 
   \begin{equation}\label{Eq. explicit value of quadratic Gauss sums}
   		G_q=(-1)^{f-1}\i^{\frac{f(p-1)^2}{4}}\sqrt{q}.
   \end{equation}

   	We need the following result.
   	
   	\begin{lemma}\label{Lem. Galois actions on quadratic gauss sums}
   		Let $q$ be an odd prime power and $G_q$ be the quadratic Gauss sum over $\mathbb{F}_q$. 
   		
   		{\rm (i)} For any $a\in\mathbb{F}_q$, we have 
   		$$\sum_{x\in\mathbb{F}_q}\phi_q(x)\zeta_p^{\Tr(ax)}=\phi_q(a)G_q.$$
   		
   		{\rm (ii)} Given any $b\in\mathbb{F}_q^{\times}$, we have 
   		$$\sum_{x\in\mathbb{F}_q}\zeta_p^{\Tr(bx^2)}=\phi_q(b)G_q.$$
   	\end{lemma}
   
   \begin{proof}
   	(i) When $a=0$, since
   	$$\sum_{x\in\mathbb{F}_q}\phi_q(x)=0,$$
   	the result clearly holds. Now suppose $a\in\mathbb{F}_q^{\times}$. Then 
   	$$\sum_{x\in\mathbb{F}_q}\phi_q(x)\zeta_p^{\Tr(ax)}=\phi_q(a)\sum_{x\in\mathbb{F}_q}\phi_q(ax)\zeta_p^{\Tr(ax)}=\phi_q(a)\sum_{x\in\mathbb{F}_q}\phi_q(x)\zeta_p^{\Tr(x)}=\phi_q(a)G_q.$$
   	
   	(ii) If $b\in\mathcal{S}_q$, then by (\ref{Eq. quadratic gauss sums with Sq}) one can verify that 
   	$$\sum_{x\in\mathbb{F}_q}\zeta_p^{\Tr(bx^2)}=\sum_{x\in\mathbb{F}_q}\zeta_p^{\Tr(x^2)}=G_q=\phi_q(b)G_q.$$
   	For the case $b\in\mathcal{N}_q$, using (\ref{Eq. quadratic gauss sums with Nq}) we obtain 
   	$$\sum_{x\in\mathbb{F}_q}\zeta_p^{\Tr(bx^2)}=1+2\sum_{x\in\mathcal{N}_q}\zeta_p^{\Tr(x)}=-G_q=\phi_q(b)G_q.$$
   	
   	In view of the above, we have completed the proof.
   \end{proof}
   
   Now we are in a position to prove our second theorem. 

  {\noindent\bfseries Proof of Theorem \ref{Thm. B}}. Recall that the symmetric matrix $N_q$ and the subset $\mathcal{H}_q$ are defined by (\ref{Eq. definition of Nq}) and (\ref{Eq. definition of Hq}) respectively.  We first consider the matrix 
  $$L_q=N_qA_q(2)N_q.$$
  Given any integers $i,j\in[1,n]$, noting that $y_k^2=s_k$ for any integer $k\in[1,n]$, we have 
  \begin{align*}
  	    L_q(i,j)
 &= \sum_{k=1}^nN_q(i,k)\sum_{r=1}^{n}A_q(2)(k,r)N_q(r,j)\\
 &=\sum_{k=1}^n\left(\zeta_p^{\Tr(2y_iy_k)}+\zeta_p^{-\Tr(2y_iy_k)}\right)\sum_{r=1}^n\sum_{z\in\mathbb{F}_q^{\times}}\zeta_p^{\Tr\left(z+\frac{s_k+s_r}{z}\right)}\left(\zeta_p^{\Tr(2y_ry_j)}+\zeta_p^{-\Tr(2y_ry_j)}\right)\\
 &=\sum_{k=1}^n\left(\zeta_p^{\Tr(2y_iy_k)}+\zeta_p^{-\Tr(2y_iy_k)}\right)\sum_{r=1}^n\sum_{z\in\mathbb{F}_q^{\times}}\zeta_p^{\Tr\left(z(s_k+s_r)+\frac{1}{z}\right)}\left(\zeta_p^{\Tr(2y_ry_j)}+\zeta_p^{-\Tr(2y_ry_j)}\right)\\
 &=\sum_{k=1}^n\left(\zeta_p^{\Tr(2y_iy_k)}+\zeta_p^{-\Tr(2y_iy_k)}\right)\sum_{r=1}^n\sum_{z\in\mathbb{F}_q^{\times}}\zeta_p^{\Tr\left(z(y_k^2+y_r^2)+\frac{1}{z}\right)}\left(\zeta_p^{\Tr(2y_ry_j)}+\zeta_p^{-\Tr(2y_ry_j)}\right)\\
 &=\sum_{z\in\mathbb{F}_q^{\times}}\zeta_p^{\Tr\left(\frac{1}{z}\right)}\sum_{k=1}^n\left(\zeta_p^{\Tr(zy_k^2+2y_iy_k)}+\zeta_p^{\Tr(zy_k^2-2y_iy_k)}\right)\sum_{r=1}^n\left(\zeta_p^{\Tr(zy_r^2+2y_jy_r)}+\zeta_p^{\Tr(zy_r^2-2y_jy_r)}\right).
  \end{align*}
  Observing that 
  $$\left\{y:\ y\in\mathcal{H}_q\right\}\cup\left\{-y:\ y\in\mathcal{H}_q\right\}=\mathbb{F}_q^{\times},$$
  we obtain 
  \begin{align*}
  	   L_q(i,j)
 &=\sum_{z\in\mathbb{F}_q^{\times}}\zeta_p^{\Tr\left(\frac{1}{z}\right)}\sum_{x\in\mathbb{F}_q^{\times}}\zeta_p^{\Tr(zx^2+2y_ix)}\sum_{y\in\mathbb{F}_q^{\times}}\zeta_p^{\Tr(zy^2+2y_jy)}\\
 &=\sum_{z\in\mathbb{F}_q^{\times}}\zeta_p^{\Tr\left(\frac{1}{z}\right)}\left(-1+\sum_{x\in\mathbb{F}_q}\zeta_p^{\Tr(zx^2+2y_ix)}\right)\left(-1+\sum_{y\in\mathbb{F}_q}\zeta_p^{\Tr(zy^2+2y_jy)}\right)\\
 &=\sum_{z\in\mathbb{F}_q^{\times}}\zeta_p^{\Tr\left(\frac{1}{z}\right)}\left(-1+\sum_{x\in\mathbb{F}_q}\zeta_p^{\Tr\left(z(x+\frac{y_i}{z})^2-\frac{y_i^2}{z}\right)}\right)\left(-1+\sum_{y\in\mathbb{F}_q}\zeta_p^{\Tr\left(z(y+\frac{y_j}{z})^2-\frac{y_j^2}{z}\right)}\right)\\
 &=\sum_{z\in\mathbb{F}_q^{\times}}\zeta_p^{\Tr\left(\frac{1}{z}\right)}\left(-1+\zeta_p^{-\Tr(y_i^2/z)}\sum_{x\in\mathbb{F}_q}\zeta_p^{\Tr(zx^2)}\right)\left(-1+\zeta_p^{-\Tr(y_j^2/z)}\sum_{y\in\mathbb{F}_q}\zeta_p^{\Tr(zy^2)}\right)
\end{align*}
  Since $z\in\mathbb{F}_q^{\times}$, applying Lemma \ref{Lem. Galois actions on quadratic gauss sums}(ii) to the above equality, we obtain 
  \begin{align*}
  	   L_q(i,j)
&=\sum_{z\in\mathbb{F}_q^{\times}}\zeta_p^{\Tr\left(\frac{1}{z}\right)}\left(-1+\zeta_p^{-\Tr(y_i^2/z)}\phi_q(z)G_q\right)\left(-1+\zeta_p^{-\Tr(y_j^2/z)}\phi_q(z)G_q\right)\\
&=\sum_{z\in\mathbb{F}_q^{\times}}\zeta_p^{\Tr(z)}\left(-1+\zeta_p^{-\Tr(y_i^2z)}\phi_q(z)G_q\right)\left(-1+\zeta_p^{-\Tr(y_j^2z)}\phi_q(z)G_q\right)\\
&=\sum_{z\in\mathbb{F}_q^{\times}}\zeta_p^{\Tr(z)}-G_q\sum_{z\in\mathbb{F}_q^{\times}}\phi_q(z)\left(\zeta_p^{\Tr((1-y_i^2)z)}+\zeta_p^{\Tr((1-y_j^2)z)}\right)+G_q^2\sum_{z\in\mathbb{F}_q^{\times}}\zeta_p^{\Tr((1-y_i^2-y_j^2)z)}.
  \end{align*}
  Now using (\ref{Eq. b in the proof of Thm. A}) and Lemma \ref{Lem. Galois actions on quadratic gauss sums}(i), for any integers $i,j\in[1,n]$, we finally obtain 
  \begin{align}\label{Eq. a in the proof of Thm. B}
  	    L_q(i,j)
  &=-1-G_q^2\left(\phi_q(1-y_i^2)+\phi_q(1-y_j^2)\right)+G_q^2\sum_{z\in\mathbb{F}_q^{\times}}\zeta_p^{\Tr((1-y_i^2-y_j^2)z)}\notag\\
  &=-1-G_q^2\left(1+\phi_q(1-y_i^2)+\phi_q(1-y_j^2)\right)+G_q^2\delta_q(i,j),
  \end{align}
  where 
  $$\delta_q(i,j)=
  \begin{cases}
  	q  & \mbox{if}\ y_i^2+y_j^2=1,\\
  	0  & \mbox{otherwise.}
  \end{cases}$$
  
  Next we use (\ref{Eq. a in the proof of Thm. B}) to construct an additive decomposition of $L_q$. Define an $n\times 1$ matrix 
  $$\u_q=\left(t_1(q),t_2(q),\cdots,t_n(q)\right)^T,$$
  where 
  $$t_k(q)=\frac{1+G_q^2}{2}+G_q^2\cdot \phi_q(1-y_k^2)$$
  for any integer $k\in[1,n]$. Also, let the $n\times 1$ matrix 
  $$\1_n=\left(1,1,\cdots,1\right)^T.$$
  Then, for any integers $i,j\in[1,n]$, one can verify that 
  $$\left(\u_q\1_n^T+\1_n\u_q^T\right)(i,j)=1+G_q^2\left(1+\phi_q(1-y_i^2)+\phi_q(1-y_j^2)\right).$$
  Combining this with (\ref{Eq. a in the proof of Thm. B}), we obtain the additive decomposition
  \begin{equation}\label{Eq. b in the proof of Thm. B}
  	L_q=N_qA_q(2)N_q=G_q^2\left[\delta_q(i,j)\right]_{1\le i,j\le n}-\u_q\1_n^T+\1_n\u_q^T.
  \end{equation}
  
  Now we focus on $\rank{A_q(2)}$. Note that 
  $$\#\left\{i\in[1,n]:\ y_i^2+y_j^2=1\right\}\le 1$$
   for any integer $j\in[1,n]$, and that $y_1^2,y_2^2,\cdots,y_n^2$ are distinct. Thus, the non-zero entries of $[\delta_q(i,j)]_{1\le i,j\le n}$ are all in distinct rows and distinct columns. This, together with Lemma \ref{Eq. number of solutions of an equation over finite field}, implies that 
   	$$\rank{G_q^2\left[\delta_q(i,j)\right]_{1\le i,j\le n}}\le \#\left\{(x,y)\in\mathcal{H}_q\times\mathcal{H}_q:\ x^2+y^2=1\right\}=\frac{1}{4}\left(q-4+(-1)^{n+1}\right).$$
   Since $N_q$ is invertible by Lemma \ref{Lem. Nq is an invertible matrix}, applying the above inequality to (\ref{Eq. b in the proof of Thm. B}), we obtain 
   \begin{align*}
   	\rank{A_q(2)}=\rank{L_q}\le \frac{1}{4}\left(q-4+(-1)^{n+1}\right)+2=\frac{1}{4}\left(q+4+(-1)^{n+1}\right).
   \end{align*}
   By this we immediately obtain 
   $$\rank{A_q(2)}<n$$
   for any $q\ge 11$, that is, $A_q(2)$ is singular whenever $q\ge 11$. For $q=3,5,7$, with the help of a computer, we have $\det A_3(2)=2, \det A_5(2)=-5$, and $\det A_7(2)=49$. 
   
   In view of the above, we have completed the proof. \qed 
   
   {\bfseries Declaration of competing interest} 
   
   The author declares that he has no conflict of interest.
   
   {\bfseries Data availability}
   
   No data was used for the research described in the article.

\end{document}